\documentclass[11pt]{article}
\usepackage{amsfonts}
\usepackage{amssymb}
\usepackage{amsmath}
\usepackage{mathrsfs}
\usepackage{enumerate}
\usepackage{color}
\usepackage{tikz}
\usepackage{mathrsfs,amscd,amssymb,amsthm,amsmath,bm,graphicx,psfrag,subfigure,url}
\usepackage[titletoc]{appendix}
\usepackage{epstopdf}
\usepackage{graphicx}
\usepackage{enumitem}
\makeatletter
\newcommand*\bigcdot{\mathpalette\bigcdot@{.5}}
\newcommand*\bigcdot@[2]{\mathbin{\vcenter{\hbox{\scalebox{#2}{$\m@th#1\bullet$}}}}}
\makeatother

\usepackage{hyperref} 

\usepackage{indentfirst}

\def \[{\begin{equation}}
\def \]{\end{equation}}

\newtheorem{thm}{Theorem}[section]

\newtheorem{claim}{Claim}

\newtheorem{lem}[thm]{Lemma}
\newtheorem{cor}[thm]{Corollary}

\newtheorem{pb}[thm]{Problem}

\begin{document}

\begin{center}
{\Large \bf Spectral extremal graphs 
without intersecting triangles \\ as a minor}
\vspace{4mm}

{\large Xiaocong He$^1$, Yongtao Li$^{1,2}$,  Lihua Feng$^1$\footnote{Corresponding author. 
This work was supported by NSFC (Nos. 12271527, 12071484, 11931002), Natural Science Foundation
of Hunan Province (Nos. 2020JJ4675, 2021JJ40707) and the Fundamental Research Funds for the Central Universities of Central South University (Grant No. 2021zzts0034).  \\ \hspace*{5mm}{Email addresses}: hexc2018@qq.com (X. He), 
ytli0921@hun.edu.cn (Y. Li), fenglh@163.com (L. Feng).}}
\vspace{3mm}

$^1$School of Mathematics and Statistics, HNP-LAMA, Central South University \\
Changsha, Hunan, 410083, P.R. China \\ 
{ $^2$School of Mathematics, Hunan University} \\
{ Changsha, Hunan, 410082, P.R. China }  \\
\end{center}


\noindent {\bf Abstract}: \
Let $F_s$ be the friendship graph obtained from $s$ triangles by sharing a common vertex. 
For fixed $s\ge 2$ and sufficiently large $n$, 
the $F_s$-free graphs of order $n$ which attain the maximal spectral radius was firstly characterized by Cioab\u{a}, Feng, Tait and Zhang [Electron. J. Combin. 27 (4) (2020)],
and later uniquely determined by Zhai, Liu and Xue [Electron. J. Combin. 29 (3) (2022)].  
Recently, the spectral extremal problems was widely studied 
for graphs containing no $H$ as a minor. 
For instance, Tait [J. Combin. Theory Ser. A 166 (2019)], 
Zhai and Lin [J. Combin. Theory Ser. B 157 (2022)]
solved the case $H=K_r$ and $H=K_{s,t}$, respectively. 
Motivated by these results, 
we consider the spectral extremal problems  
in the case $H=F_s$. 
We shall prove that  $K_s \vee I_{n-s}$ is the unique graph that attain the maximal spectral radius over all $n$-vertex $F_s$-minor-free graphs.
Moreover, let $Q_t$ be the graph obtained from $t$ copies of 
the cycle of length $4$  
by sharing a common vertex. 
We also determine the unique $Q_t$-minor-free graph attaining the maximal spectral radius. Namely, 
$K_t \vee M_{n-t}$, where $M_{n-t}$ 
is a graph obtained from an independent set of order $n-t$  
by embedding a matching consisting of $\lfloor \frac{n-t}{2}\rfloor$ 
edges.

\vspace{2mm} \noindent{\bf Keywords}: Spectral radius; Intersecting cycles; Minor-free; Extremal graph. 

\vspace{2mm} \noindent{\bf AMS classification}: 05C50, 05C35.
\vspace{2mm}

\setcounter{section}{0}
\section{Introduction} 
\setcounter{equation}{0}

Let $G$ be a graph with vertex set $V(G)=\{v_1,\dots,v_n\}$ and edge set $E(G)=\{e_1,e_2,\ldots ,e_m\}$. 
We write $G-v$ for the graph obtained from $G$ by deleting vertex $v \in V(G)$ and its incident edges, 
and $G-uv$ for the graph by deleting the edge $uv \in E(G)$.
 This notation is naturally extended if more than one vertex or edge is deleted. Similarly, $G+uv$ is obtained from $G$ by adding an edge $uv\notin E(G)$. As usual, 
a \textit{complete graph} on $n$ vertices is denoted by $K_n$, and an {\it independent set} on $n$ vertices 
is denoted by $I_n$. 
We write $P_t$ for a {\it path} on $t$ vertices. 
Commonly, we  refer to a path by the nature sequence of its vertices, say $P_t=x_1x_2\ldots x_t$, and call $P_t$ a path starting from $x_1$ to $x_t$. 
In addition, we write $C_t$ for a {\it cycle} on $t$ vertices, 
and write $K_{s,t}$ for the \textit{complete bipartite graph} 
with two parts of sizes $s$ and $t$. 
We denote by $G[X, Y]$ the  \textit{bipartite subgraph} with bipartition ($X, Y$) whose 
edges are that of $G$ between $X$ and $Y$.

The \textit{adjacency matrix} $A(G)=(a_{ij})$ of $G$ is an $n\times n$ matrix with $a_{ij}=1$ if $v_i$ is adjacent to $v_j$, and $0$ otherwise. The \textit{spectral radius} of $G$ is the largest eigenvalue of $A(G)$, which is denoted by $\rho(G)$. For each vertex $v$ in $G$, let $N_G(v):=\{u\in V(G): uv\in E(G)\}$ and $N_G[v]:=N_G(v)\cup\{v\}$. The \textit{degree} of $v$ is denoted by $d_G(v)=|N_G(v)|$. Similarly, for each subgraph $H$ of $G$, let $N_G(H)$ be the set of vertices in $V(G)\setminus V(H)$ that are adjacent to some vertex of $H$.

A \textit{clique} of $G$ is a subset $S$ of $V(G)$ such that $G[S]$ is a complete subgraph. Let $G_1$ and $G_2$ be two disjoint graphs, denote by $G_1\cup G_2$ the \textit{vertex-disjoint union} of $G_1$ and $G_2$. 
For simplicity, we write $sG$ for the vertex-disjoint union 
of $s$ copies of $G$. 
The \textit{join} $G_1\vee G_2$  is obtained from $G_1\cup G_2$ by joining each vertex of $G_1$ to each vertex of $G_2$. Moreover, we denote by $M_t$ the disjoint union of $\lfloor\frac{t}{2}\rfloor$ copies of $K_2$ and $\lceil\frac{t}{2}\rceil-\lfloor\frac{t}{2}\rfloor$ isolated vertex (maybe no isolated vertex). For example, if $t=2s$ for some $s\ge 1$, then $M_t=sK_2$; 
if $t=2s+1$, then $M_s=K_1 \cup sK_2$. 
For graph notation and terminology undefined here, readers are referred to \cite{BM}.

\subsection{Spectral extremal graphs for friendship graphs}

A graph $G$ is called $H$-free if $H$ is not a subgraph of $G$. In 2010, Nikiforov \cite{Nikiforov10} proposed a spectral extremal problem, which is now known as Brualdi-Solheid-Tur\'{a}n type problem. 
More precisely, what is the maximum spectral radius 
among all $n$-vertex $H$-free graphs?  
In the past few decades, the problem has been investigated by many researchers for various graphs $H$, such as, the complete graphs \cite{NikiforovKr,wilf,LP2022second}, 
the complete bipartite graphs \cite{babai,za}, the books and theta graphs \cite{zhaigraph}, 
the friendship graphs \cite{cioaba2,ZLX2022,ZHG21}, 
the intersecting odd cycles \cite{LP2021,CLZ2021}, 
the intersecting cliques \cite{DKLNTW2021}, 
the paths and linear forests \cite{Nikiforov10,chen}, 
the odd wheels \cite{cioaba1}, 
the quadrilateral \cite{NikiforovKr,zhaic4}, 
the hexagon \cite{zhaic6}, 
the short odd cycles \cite{LG2021,LSY2022,LP2022oddcycle}, 
the square of a path \cite{zyh}, 
the fan graph \cite{WZ2022dm}. 
We refer the reader to \cite{NikifSurvey} for a comprehensive survey.

Let $F_s$ be the graph 
obtained from $s$ triangles by intersecting in exactly one common vertex. 
Under the definition of join, we have $F_s= K_1 \vee s K_2$.  
The graph $F_s$ is also known as 
the friendship graph because it is the only extremal graph in the
famous Friendship Theorem \cite[Chapter 43]{AZ2014}, 
which asserts that if $G$ is a graph on $n$ vertices such that 
any two distinct vertices have exactly one common neighbor, 
then $n$ is odd and $G$ consists of $\frac{n-1}{2}$ triangles 
intersecting in a common vertex. 
The extremal problem involving $F_s$ was widely studied in the literature. 
Tracing back to 1995, Erd\H{o}s, F\"{u}redi,
Gould and Gunderson \cite{Erdos95} proved  the following
result.

\begin{thm}[Erd\H{o}s--F\"{u}redi--Gould--Gunderson \cite{Erdos95}, 1995]
\label{thmErdos95}
Let $s \geq 1$ and $n\geq 50s^2$ be positive integers. 
If $G$ is an $F_s$-free graph on $n$ vertices, then 
\begin{equation*}
e(G) \le \left\lfloor \frac {n^2}{4}\right \rfloor+ \left\{
  \begin{array}{ll}
   s^2-s, \quad~~  \mbox{if $s$ is odd,} \\
    s^2-\frac{3}{2} s, \quad \mbox{if $s$ is even}.
  \end{array} \right. 
\end{equation*}
\end{thm}

We write $\mathrm{Ex}(n,F_s)$ for
 the set of $n$-vertex $F_s$-free graphs
with the maximum number of edges.  
The extremal graphs in $\mathrm{Ex}(n,F_s)$ 
were also characterized  by  Erd\H{o}s, F\"{u}redi,
Gould and Gunderson \cite{Erdos95}. 
Roughly speaking, they are constructed from 
the balanced complete bipartite graph $K_{\lfloor \frac{n}{2}\rfloor, \lceil \frac{n}{2} \rceil }$ by embedding a ``small 
graph'' in any one side. 
More precisely, for odd $s$,
the  graphs $G\in \mathrm{Ex}(n,F_s)$ are obtained from 
 $K_{\lfloor \frac{n}{2}\rfloor, \lceil \frac{n}{2} \rceil }$
by embedding two vertex-disjoint copies of the complete graph $K_s$ in one side. 
 For even $s$, the extremal graphs
 are constructed from $K_{\lfloor \frac{n}{2}\rfloor, \lceil \frac{n}{2} \rceil }$ by embedding
 a graph with $2s-1$ vertices, $s^2-\frac{3}{2} s$ edges and maximum degree $s-1$ in one side.   
 Here, we remark  that for even $s$, the embedded graph is a nearly $(s-1)$-regular graph
on $2s-1$ vertices with degree
 sequence $(s-1,\ldots ,s-1,s-2)$. 
 It is known \cite{BM} that such a graph does exist 
 and it is not unique for every even $s\ge 2$.

In 2020, 
the spectral version of Theorem \ref{thmErdos95} was studied by 
Cioab\u{a}, Feng, Tait and Zhang \cite{cioaba2}. 
They characterized the spectral extremal $F_s$-free graphs. 
More precisely, they proved that for fixed $s\ge 2$ and sufficiently large $n$, if $G$ is an $F_s$-free graph of  order $n$ with maximal spectral radius, then $G$ attains the maximum number of edges over all $n$-vertex $F_s$-free graphs.

 \begin{thm}[Cioab\u{a}--Feng--Tait--Zhang \cite{cioaba2}, 2020]    \label{thmCFTZ20}
Let $s\ge 2$ and $G$ be an $F_s$-free graph on $n$ vertices.
For  sufficiently large $n$, if $G$ has the maximal spectral radius, then
$$  G \in  \mathrm{Ex}(n,F_s). $$
\end{thm}

In 2022, Zhai, Liu and Xue \cite{ZLX2022} provided a further  characterization of $G$ and
determined  the {\it unique} spectral extremal graph of $F_s$ 
for sufficiently large $n$. 
In other words, they determined the unique embedded 
subgraph in the extremal graph of $\mathrm{Ex}(n,F_s)$. 
Let $H^*$ be the graph of order $2s-1$ 
with vertex set $V(H^*)=\{w_0\} \cup A \cup B$
such that $N(w_0)=A$ and $|B|=|A| +2=s$. 
Then we partition $A$ into $A_1\cup A_2$,
and $B$ into $\{u_0\}\cup B_1\cup B_2$
such that $|A_1|=|A_2|=|B_2|=\frac{s-2}{2}$ and $|B_1|=\frac{s}{2}$.
Finally, we join $s-1$ edges from $u_0$ to $A_1\cup B_1$,
$\frac{s-2}{2}$ independent edges between $B_2$ and $A_2$,
and some additional edges such that both $A$ and $B_1\cup B_2$ are cliques. 

\begin{thm}[Zhai--Liu--Xue \cite{ZLX2022}, 2022] 
Let $s\ge 2$ and $G$ be an $F_s$-free graph with 
the maximal spectral radius. Then for sufficiently large $n$, 
the graph $G$ is obtained from $K_{\lfloor \frac{n}{2}\rfloor, \lceil \frac{n}{2} \rceil }$ by embedding a graph $H$ in the part of size
$\lfloor \frac{n}{2}\rfloor$, where $H=K_s \cup K_s$ if $s$ is odd;  and $H=H^*$ if $s$ is even. 
\end{thm}

We refer the readers to \cite{ZLX2022} for more details 
and \cite{WKX2023} for a generalization of Theorem \ref{thmCFTZ20}. 
Furthermore, 
Li, Lu and Peng \cite{2022LLP} 
got rid of the condition that $n$ is sufficiently large if $s=2$. 
They proved that the unique $n$-vertex 
$F_2$-free spectral extremal graph is  the balanced complete bipartite graph adding an edge in the vertex part with  smaller size if $n\ge 7$.
Moreover,  it was also proved in \cite{2022LLP} that the
unique $m$-edge $F_2$-free spectral extremal graph is the join of $K_2$
with an independent set of $\frac{m-1}{2}$ vertices if $m\ge 8$, and the conditions $n\ge 7$ and $m\ge 8$ are tight.

\subsection{Spectral extremal graphs forbidding minors}

A graph $H$ is a \textit{minor} of $G$ if a graph isomorphic to $H$ can be obtained from $G$ by the following operations: vertex deletion, edge deletion and edge contraction. A graph $G$ is \textit{$H$-minor-free} if it does not contain $H$ as a minor.

A natural question at the intersection of graph minor theory and Brualdi-Solheid-Tur\'{a}n type problem asks, for a given graph $H$, what is the maximal spectral radius over all $n$-vertex graphs which do not contain $H$ as a minor? 
Indeed, such a problem was recently becoming popular and 
some elegant results have been published in the following two aspects. 

There are two famous conjectures in the study of spectral extremal problems on planar and outerplanar graphs. 
It is known that a graph is planar if and only if it is $\{K_5,K_{3,3}\}$-minor-free. Moreover, a graph is outerplanar if and only if it is $\{K_4,K_{2,3}\}$-minor-free; 
see, e.g., \cite{BM}.  In 1990, Cvetkovi\'c and Rowlinson \cite{Cvetkovi} conjectured that the join graph $K_1+P_{n-1}$ is the unique graph attaining the maximal spectral radius over all outerplanar graphs of order $n$. 
In 1991, Boots and Royle \cite{Boots}, and independently Cao and Vince \cite{cao}, proposed a spectral problem 
for planar graphs, which conjectured that 
$K_2+P_{n-2}$ is the unique graph attaining 
the maximal spectral radius over all planar graphs  of order $n$. 
Many scholars contributed to these two conjectures. 
In particular, Tait and Tobin \cite{tait2017} confirmed these  conjectures for sufficiently large $n$. In 2021, Lin and Ning \cite{lhq} confirmed Cvetkovi\'{c}-Rowlinson conjecture for all $n\ge 2$ except for $n=6$. 

Apart from the planar and outerplanar graphs, 
it is natural to consider the maximal spectral radius 
for $H$-minor-free graphs with a specific graph $H$. 
In particular, setting $H$ as the complete graph $K_r$ and the complete bipartite graph $K_{s,t}$.  
In 2004, Hong \cite{Hong} determined the extremal graph with maximum spectral radius for $K_5$-minor-free graphs. In 2017, Nikiforov \cite{Nikiforov2017} obtained a sharp upper bound on the spectral radius of $K_{2,t}$-minor-free graphs. In 2019, Tait \cite{Tait2019} characterized the spectral extremal graphs with no $K_r$ as a minor.  
In 2022, Zhai and Lin \cite{zhaijctb} determined  completely 
the spectral extremal graphs for $K_{s,t}$-minor.

Comparing with the rich development of 
the traditional spectral extremal problem (listed in Subsection 1.1), 
there are few spectral results for minor-free graphs  
although the spectral problem of minors 
has risen in popularity in the past few years.  
Recall that $F_s$ is the friendship graph consisting of $s$ triangles sharing a common vertex. 
As stated in previous subsection, the traditional problem 
for the friendship graph $F_s$ 
has recently received extensive attention and investigation; 
see, e.g., \cite{cioaba2,ZLX2022,WKX2023}. 
Inspired by the results of $K_r$-minor and $K_{s,t}$-minor, 
we shall present one more result for 
$H$-minor-free problem by putting $H=F_s$. 
We will determine the largest spectral radius of a graph over all $F_s$-minor-free graphs of order $n$, 
and we show that $K_s\vee I_{n-s}$ is 
the unique spectral extremal graph.

\begin{thm}\label{thm1}
Let $s\ge 1$ be an integer and $G$ be an $F_s$-minor-free graph of order $n$. 
Then for sufficiently large $n$, we have 
$$\rho(G)\leq \rho(K_{s}\vee {I_{n-s}}),$$
with equality if and only if $G =  K_{s}\vee {I_{n-s}}$.
\end{thm}

 Moreover, we would like to consider a more general problem 
 for intersecting cycles, rather than triangles. 
In particular, we shall study the spectral extremal problem
 for graphs with no intersecting quadrilaterals (cycles with length $4$) 
 as a minor. 
  Let $Q_t$ be the graph obtained from $t$ copies of $C_4$   
by sharing a common vertex. 
The second result in this paper determines the unique spectral extremal graphs among 
all $Q_t$-minor-free graphs. 
Recall that $M_{n-t}$ can be obtained from 
an independent set on $n-t$ vertices by embedding a maximal 
matching. 
In other words, we have $M_{n-t}= \frac{n-t}{2}K_2$ if $n-t$ is even; 
and $M_{n-t}=K_1\cup \frac{n-t-1}{2}K_2$ if $n-t$ is odd.

\begin{thm}\label{thm2}
Let $t\ge 1$ be an integer and 
$G$ be a $Q_t$-minor-free  graph of order $n$. 
Then for sufficiently large $n$, we have
$$\rho(G)\leq \rho(K_t\vee M_{n-t}),$$
with equality if and only if $G =  K_t\vee M_{n-t}$.
\end{thm}

The rest of this paper is organized as follows. 
In Section 2, some preliminaries are presented for our purpose. 
In Section 3, we will give some structural properties of $F_s$-minor-free graphs. 
In Sections 4, we shall give the details of the proof of Theorem \ref{thm1}. 
In Section 5 and 6, we will 
characterize the structure of $Q_t$-minor-free graphs, and then we present the proof of Theorem \ref{thm2}. 
The techniques used in our proof are mainly inspired 
by Tait \cite{Tait2019}.

\section{Preliminary}
Mader \cite{Mader} proved an elegant result on the number of edges in $H$-minor-free graphs.

\begin{lem}[Mader \cite{Mader}, 1967]   \label{e-minor1}
Let $G$ be an $n$-vertex graph. For every graph $H$, if $G$ is $H$-minor-free, then there exists a constant $C>0$ such that $$e(G)\leq Cn.$$
\end{lem}

The following lemma has been proved many times in the literature; see, e.g., \cite{Thomason2007, 47}.  

\begin{lem} \label{lem-22}
Let $G$ be a bipartite graph on $n$ vertices 
with no $K_{s,t}$-minor and vertex partition $A$ and $B$. 
Let $|A|=a$ and $|B|=n-a$. Then there is a constant $C>0$ 
depending only on $s$ and $t$ such that  
\begin{equation*}
  e(G) \le Ca + (s-1)n. 
  \end{equation*}
\end{lem}

The following lemma was an implicative result, which can be seen from the proof of \cite[Lemma 10]{tait2017} and \cite[Claim 3.4]{Tait2019} as well. For completeness, we include a detailed proof. 

\begin{lem} \label{lem-27}
If $\mathbf x=(\mathrm{x}_u)_{u\in V(G)}$ is a positive eigenvector with the maximum entry $1$ which corresponds to $\rho(G)$, then $\mathrm{x}_u\geq \frac{1}{\rho(G)}$ for all $u\in V(G)$. 
\end{lem}

\begin{proof}
 Let $w\in V(G)$ such that $\mathrm{x}_w=1$.
If $u=w$, then $\mathrm{x}_u=1\geq \frac{1}{\rho(G)}$. So next we suppose that $u\neq w$. We consider the following two cases.

{\bf Case~1.} $u$ is adjacent to $w$. By the eigenvector-eigenvalue equation of $G$ on $u$, $$\rho(G)\mathrm{x}_u=\sum_{uv\in E(G)}\mathrm{x}_v\geq \mathrm{x}_w=1,$$
which implies that $\mathrm{x}_u\geq \frac{1}{\rho(G)}.$

{\bf Case~2.} $u$ is not adjacent to $w$. Let $G'$ be the graph obtained from $G$ by deleting all edges incident with $u$ and adding an edge $uw$.
Note that $uw$ is a pendant edge. Clearly, $G'$ is $I_{s,t}$-minor-free.
Hence
\begin{eqnarray*}
  0 &\geq& \rho(G')-\rho(G)\geq \frac{\mathbf x^{\mathrm{T}}A(G')\mathbf x}{\mathbf x^{\mathrm{T}}\mathbf x}- \frac{ {\mathbf x^{\mathrm{T}}} A(G)\mathbf x}{\mathbf x^{\mathrm{T}}\mathbf x}\\
   &=& \frac{2\mathrm{x}_u}{\mathbf x^{\mathrm{T}}\mathbf x} \Big(\mathrm{x}_w-\sum_{uv\in E(G)}\mathrm{x}_v\Big)\\
 &=&  \frac{2\mathrm{x}_u}{\mathbf x^{\mathrm{T}}\mathbf x} \Big(1-\sum_{uv\in E(G)}\mathrm{x}_v \Big),
\end{eqnarray*}
which implies that $$\sum_{uv\in E(G)}\mathrm{x}_v\geq1.$$
By  the eigenvector-eigenvalue equation of $G$ on $u$, $$\rho(G) \mathrm{x}_u=\sum_{uv\in E(G)}\mathrm{x}_v\geq1,$$ which yields that $$\mathrm{x}_u\geq \frac{1}{\rho(G)}.$$
So we finish the proof.
\end{proof}

\section{Structure of graphs with no $F_{s}$-minor}

In this subsection, we will present 
some lemmas for $F_{s}$-minor-free graphs.  

\begin{lem}\label{e-minor3}
Let $G$ be an $n$-vertex $F_{s}$-minor-free bipartite graph with vertex partition $A$ and $B$. If $|A|=a$ and $|B|=n-a$, then there exists a constant $C>0$ depending only on $s$ such that
$$e(G)\leq Ca+ sn.$$
\end{lem}

\begin{proof}
Suppose that $G$ is $F_s$-minor-free. 
Note that any $K_{s+1,s}$-minor contains an $F_s$-minor. 
This yields that $G$ is $K_{s+1,s}$-minor-free. 
Hence, the assertion follows from Lemma \ref{lem-22}.
\end{proof}

\begin{lem}\label{spec1-minor3}
Let $G$ be an $n$-vertex graph with the maximum spectral radius $\rho(G)$ among all $F_{s}$-minor-free connected graphs. Then $\rho(G)\geq \sqrt{s(n-s)}$. 
\end{lem}

\begin{proof}
Observe that $K_{s,n-s}$ is $F_{s}$-minor-free and $\rho (K_{s,n-s})=\sqrt{s(n-s)}$, as desired. 
\end{proof}

Before showing our results, we fix some notions firstly. 
A subset $S$ of $V(G)$ is called a \textit{fragment} if 
the induced subgraph $G[S]$ is connected, where $G[S]$ is the subgraph induced by $S$. Distinct fragments $S'$ and $S''$ are said to be {\it adjacent} if 
there exist two vertices  $u'\in S'$ and $u''\in S''$
such that $u'u''\in E(G)$. 

\begin{lem}\label{L-minor2}
Let $G$ be an $n$-vertex $F_s$-minor-free graph. If $G$ contains a complete bipartite subgraph $K_{s,(1-\delta)n}=[A, B]$ with $|A|=s$ and $|B|=(1-\delta)n\geq 2s$, then \\
(i) $G[B]$ is $P_2$-free, and $|N_G(v)\cap B|\leq1$ 
for any $v\in V(G)\backslash (A\cup B)$; \\
(ii) There are at least $(1-2\delta)n$ vertices in $B$ which have no neighbors in $V(G)\backslash (A\cup B)$.
\end{lem}

\begin{proof}
We first prove that $G[B]$ is $P_2$-free. In fact, if there exists a path $P_2$ in $G[B]$, then $G[A\cup B]$ contains 
a subgraph consisting of one triangle and $s-1$ copies of $C_4$ by sharing a common vertex. 
By contracting an edge in each copy of $C_4$, 
we observe that $F_s$ is a minor of $G$, which is a contradiction. Hence, $G[B]$ is $P_2$-free.
Furthermore, we have the following claim.

\medskip 
{\bf Claim.} If $H$ is a component of $G-(A\cup B)$, then $|N_G(H)\cap B|\leq1$.

\begin{proof}[Proof of Claim]
Suppose that there are two vertices $u,v\in N_G(H)\cap B$. Let $G'$ be obtained from $G$ by contracting $G[\{v\}\cup V(H)]$ to a single vertex $v'$. Then $K_{s,(1-\delta)n}= [A,(B\backslash\{v\})\cup\{v'\}]$ and $P_2=  uv'$ are subgraphs of $G'$. Note that 
the intersecting subgraph consisting of one triangle and 
$s-1$ copies of $C_4$ is contained in the union of $K_{s,(1-\delta)n}= [A,(B\backslash\{v\})\cup\{v'\}]$ and $P_2=  uv'$. Hence, $F_s$ is a minor of $G$, which is a contradiction. So the claim holds.
\end{proof}

By the above Claim, we know that Part (i) holds immediately.

Now let $R=V(G)\backslash (A\cup B)$ and $D=\{v\in B: N_G(v)\cap R=\emptyset \}$. By the definition of $R$,
$$|R|=n-|A|-|B|\leq n-s-(1-\delta)n<\delta n,$$
which implies that
$R$ has at most $ \delta n$ components. By Claim, $B\backslash D$ has at most $ \delta n$ vertices.
Hence, $$|D|=|B|- |B\backslash D|\geq(1- \delta) n-\delta n=(1-2\delta) n.$$

This completes the proof of Part (ii).
\end{proof}

\begin{lem}\label{edge-minor-2}
Let $G$ be an $n$-vertex $F_s$-minor-free graph. Suppose $G$ contains a complete bipartite subgraph $K_{s,(1-\delta)n}=[A,B]$ with $|A|=s$, $|B|=(1-\delta)n$ and $(1-2\delta)n\geq2s+1$. Let $G^*$ be obtained from $G$ by adding edges to $A$ to make it a clique. Then $G^*$ is also  $F_s$-minor-free.
\end{lem}
\begin{proof}
Denote by $R=V(G)\backslash (A\cup B)$ and $D=\{v\in B: N_G(v)\cap R=\emptyset \}$. By Lemma \ref{L-minor2}, $|D|\geq(1- 2\delta) n$. Suppose that $G^*$ contains an $F_s$-minor. Then there exist $2s+1$ disjoint fragments $S_0,S_1,\ldots,S_{2s}\subseteq V(G^*)=V(G)$ with the following properties:

$(a)$\ There is at least one edge between $S_0$ and $S_{i}$ for all $i=1,\ldots,2s$.

$(b)$\ There is at least one edge between $S_{2i-1}$ and $S_{2i}$ for $i=1,\ldots,s$.

$(c)$\ There is an integer $j$ such that $S_j\cap D\neq\emptyset$, $j\in\{0,1,\ldots,2s\}$.

In fact, if $S_0,S_1,\ldots,S_{2s}\subseteq (B\setminus D)\cup R$, then $S_0,S_1,\ldots,S_{2s}$ in $G$ form an  $F_s$-minor, which is a contradiction. Hence, there exists a set $S_j$ such that $S_j\cap (A\cup D)\not=\emptyset$ for $j=0,1,\ldots,2s$. If $S_j\cap D\not=\emptyset$, then we are done. Otherwise, we have $S_j\cap A\not=\emptyset$. Furthermore, we can suppose that $S_i\cap D=\emptyset$ for all $0\leq i\leq 2s$. Then choose one vertex $u\in D$ and let $S_j^{\prime}=S_j\cup \{u\}$. Then
$S_0,\ldots,S_j',\dots,S_{2s}$ satisfying $(a), (b)$ and $(c)$.

Let
  $$f(S_0,S_1,\ldots,S_{2s})=\left|\{S_i:S_i\cap D\neq\emptyset ~\text{for}~ i=0,1,\dots,2s\}\right|.$$
Hence, we can choose $2s+1$ disjoint fragments $S_0,S_1,\ldots,S_{2s}$ satisfying $(a), (b)$ and $(c)$ such that $f(S_0,S_1,\ldots,S_{2s})$ is as large as possible.

For $i=0,1,\ldots,2s$, if $|S_i\cap D|\ge 2$, choose a vertex $u_i\in S_i\cap D$ and let $U_i=(S_i\backslash D)\cup\{u_i\}$. If $|S_i\cap D|\leq1$, let $U_i=S_i$.
\setcounter{claim}{0}
\begin{claim}\label{claim7}
For $i=0,1,\ldots,2s$, 
the induced subgraph $G^*[U_i]$ is connected.
\end{claim}
\begin{proof}
For any two vertices $u,v$ in $U_i$, there exists a path $P$ from $u$ to $v$ in $G^*[S_i]$ since $G^*[S_i]$ is connected. If $P$ contains a vertex $w\in(S_i\cap D)\setminus\{u_i\}$, then there exist two vertices $w_1,w_2$ in $P$ such that $\{ww_1,ww_2\}\subseteq E(P)$. Since $G[B]$ is $P_2$-free, $w_1,w_2\in A$. So $w_1$ is adjacent to $w_2$. So there is a path in $G^*[S_i]$ from $u$ to $v$ containing no $w$. Hence, $G^*[U_i]$ is connected, a contradiction.
\end{proof}
\begin{claim}\label{claim8}
If $S_i$ and $S_j$ are adjacent in $G^*$ such that $|U_i\cap D|=1$ and $|U_j\cap D|\leq1$, then $U_i$ and $U_j$ are adjacent in $G^*$.
\end{claim}
\begin{proof}
Suppose on the contrary that there are no edges between $U_i$ and $U_j$ in $G^*$. Then $S_j\cap A=\emptyset$ and all the edges between $S_i$ and $S_j$ in $G^*$ have one endpoint in $(S_i\cap D)\backslash\{u_i\}$ or $(S_j\cap D)\backslash\{u_j\}$. Hence, $S_j\cap B\neq\emptyset$. We claim that $S_i\cap A=\emptyset$. Otherwise, since there are no edges between $U_i$ and $U_j$ in $G^*$, we have $S_j\cap B=S_j\cap D$. Then $|U_j\cap D|=1$. Thus, there is at least one edge between $U_i\cap A=S_i\cap A$ and $U_j\cap D$ in $G^*$, a contradiction. Suppose $uv$ is an edge with $u\in(S_i\cap D)\backslash\{u_i\}$ and $v\in S_j$, then we have $v\in S_j\cap A$, contradicting to $S_j\cap A=\emptyset$. On the other hand, suppose $uv$ is an edge with $u\in(S_j\cap D)\backslash\{u_j\}$ and $v\in S_i$, then we have $v\in S_i\cap A$, contradicting to $S_i\cap A=\emptyset$. Hence, there is at least one edge between $U_i$ and $U_j$ in $G^*$.
\end{proof}
\begin{claim}\label{claim9}
If $U_i\cap A\neq\emptyset$, then $U_i\cap D\neq \emptyset$ for $0\leq i\leq 2s$.
\end{claim}
\begin{proof}
By Claims \ref{claim7} and \ref{claim8}, disjoint fragments $U_0,U_1, \ldots, U_{2s}$ satisfy $(a), (b)$ and $(c)$. Now suppose that there exists $0\le j\le 2s$ such that $U_j\cap A\neq \emptyset$ and $U_j\cap D= \emptyset$. Then choose a vertex $w\in D\backslash \cup_{i=0}^{2s} U_i$ and let $V_j=U_j\cup\{w\}$ and $V_i=U_i$ for $0\le i\neq j\le 2s$. It is easy to see that the $2s+1$ disjoint fragments $V_0,V_1, \ldots, V_{2s}$ satisfy $(a), (b)$ and $(c)$. Moreover,
$$f(V_0,V_1, \ldots, V_{2s})=f(U_0,U_1, \ldots, U_{2s})+1=f(S_0,S_1, \ldots, S_{2s})+1,$$
which contradicts to the choice of $S_0, \ldots, S_{2s}$.
\end{proof}
\begin{claim}\label{claim10}
$G[U_i]$ is connected for $0\leq i\leq 2s$.
\end{claim}
\begin{proof}
Since $G^*[U_i]$ is connected, there exists a path $P$ from $u$ to $v$ in $G^*[U_i]$ for any two vertices $u,v \in U_i$. If $P$ contains an edge $a_1a_2$ with $a_1, a_2\in A$, then by Claim \ref{claim9}, there exists a vertex $w\in U_i \cap D$. If $w\notin V(P)$, then the edge $a_1a_2$ of $P$ may be replaced by edges $a_1w$ and $a_2w$. If $w\in V(P)$, then the subpath of $P$ containing $a_1a_2$ and $w$ may be replaced by an edge $a_1w$ or $wa_2$. The above transformations yield a path $P'$ from $u$ to $v$ which contains no edges in $G^*[A]$. So there exists a path from $u$ to $v$ in $G[U_i]$ and thus $G[U_i]$ is connected.
\end{proof}
\begin{claim}\label{claim11}
If $S_i$ and $S_j$ are adjacent in $G^*$, then $U_i$ and $U_j$ are adjacent in $G$.
\end{claim}
\begin{proof}
Suppose that there are no edges between $U_i$ and $U_j$ in $G$. There must exist two vertices $u,v$ such that $u\in U_i\cap A$ and $v\in U_j\cap A$. By Claim \ref{claim9}, there exists a vertex $w\in U_j\cap D$. Hence, there is one edge $uw$ between $U_i$ and $U_j$ in $G$. This is a contradiction.
\end{proof}
By Claims \ref{claim10} and \ref{claim11}, $U_0,U_1, \ldots, U_{2s}$ form an $F_s$-minor of $G$, which is a contradiction.

This completes the proof.
\end{proof}

\section{Proof of Theorem \ref{thm1}}

{\bf Proof of Theorem~\ref{thm1}.} 
Let $G$ be an $n$-vertex $F_s$-minor-free graph with the maximal spectral radius. 
By the choice of $G$, 
we know that $G$ must be connected. 
Then let $\mathbf x=(\mathrm{x}_u)_{u\in V(G)}$ be a positive eigenvector of $G$ corresponding to $\rho(G)$. 
We may assume by scaling that  the maximum entry of $\bf{x}$ is $\mathrm{x}_w=1$ for some $w\in V(G)$.  
We will use throughout the section that $e(G)=O(n)$ 
by Lemma \ref{e-minor1}. 
For $0<\epsilon<1$, we denote $$L=\{v\in V(G): \mathrm{x}_v> \epsilon\},$$ 
and 
$$ S=\{v\in V(G): \mathrm{x}_v\leq \epsilon\},$$
where $\epsilon$ is a small constant which will be chosen later. Clearly,
we have $V(G)=L\cup S$.

The outline of our proof is as follows: 
\begin{itemize}
\item
 Firstly, we know that 
$\lambda(G) = \Theta(\sqrt{n})$.  
Then we will show that  $|L|\le O(\sqrt{n})$ by Lemma~\ref{e-minor1}.  
Thus, we get $|S|= n- |L|=(1-o(1))n$. 
Consequently, we obtain 
$e(L) = O(|L|) \le O(\sqrt{n})$ and $e(S)\le O(n)$. 
Moreover, we can show that $e(L,S) \le (s+o(1))n$.

\item
Secondly, we shall prove that if a vertex has 
eigenvector entry close to $1$, then it has degree close to $n$; 
see Claim 2. 
Furthermore, we will show by induction 
that there are $s$ vertices in $L$ with eigenvector entry close to 1, 
and hence its degree close to $n$; see Claim 3. 

\item
Moreover, we shall show that these $s$ vertices induce a clique 
$K_{s}$; see Claim 4.

\item
Finally, we show that each of the $s$ vertices in the clique actually has degree $n-1$.
\end{itemize}

By Lemma~\ref{e-minor1}, there is a constant $C_1 :=2C>0$ such that
\begin{equation}\label{1-1}
\begin{aligned}
2e(S)\leq 2e(G)\leq C_1n.
   \end{aligned}
 \end{equation}
In addition, by Lemma~\ref{spec1-minor3}, we obtain 
\begin{equation}\label{1}
\begin{aligned}
\rho(G)\geq \sqrt{s(n-s)}.
   \end{aligned}
 \end{equation}

{\bf Claim~1.} $e(L,S)\leq (s+\epsilon)n$ and $2e(L)\leq \epsilon n$.

\begin{proof}
It is easy to see that
\begin{eqnarray*}
  \rho(G)\epsilon |L| &< & \sum_{v\in L}\rho(G) \mathrm{x}_v =   \sum_{v\in L} \sum_{z\in N_G(v)} \mathrm{x}_z
   \leq   \sum_{v\in L}d_G(v)\leq 2e(G). 
\end{eqnarray*}
Then by (\ref{1-1}) and (\ref{1}), it implies that
\begin{equation}\label{upper L}
|L|\le \frac{2e(G)}{\epsilon\rho(G)}\le \frac{C_1n}{\epsilon\sqrt{s(n-s)}}\le 
\frac{2C_1\sqrt{n}}{\epsilon\sqrt{s}}, 
\end{equation}
where the last inequality holds for sufficiently large $n$. 

By Lemma~\ref{e-minor3}, there is a constant $C_2>0$ only depending on $s$ such that
\begin{equation}\label{2-1}
\begin{aligned}
e(L,S)\leq C_2|L|+sn\leq \frac{2C_1C_2\sqrt{n}}{\epsilon\sqrt{s}}+sn\leq (s+\epsilon)n
\end{aligned}
\end{equation}
as long as $n$ is large so that 
$n \geq 4(C_1C_2)^2/  (s\epsilon^4)$. 

In addition, by (\ref{upper L}) and Lemma~\ref{e-minor1}, we have
\begin{equation}\label{2}
\begin{aligned}
2e(L)\leq C_1|L|\leq \frac{sC_1^2\sqrt{n}}{\epsilon\sqrt{s}}\leq\epsilon n
\end{aligned}
\end{equation}
as long as $n$ is sufficiently large.
So Claim~1 holds.
\end{proof}

{\bf Claim~2.}
If $u\in L$ is a vertex with $\mathrm{x}_u=1-\alpha$ 
for some constant $\alpha >0$, then there exists a constant $C_3>1$ independent of $\alpha$ and $\epsilon$ such that
$$ d_G(u)\geq [1- C_3(\alpha+\epsilon)]n.$$

\begin{proof}
Clearly, we have 
\begin{eqnarray*}
  \rho(G) \sum_{v\in V(G)}\mathrm{x}_v &=& 
  \sum_{v\in V(G)} \sum_{z\in N_G(v)}\mathrm{x}_z
  = \sum_{v\in V(G)} d_G(v)\mathrm{x}_v
  \leq \sum_{v\in L} d_G(v)+\epsilon \sum_{v\in S} d_G(v)\\
   &=&  2e(L)+\epsilon \cdot 2e(S)+(1+\epsilon)e(L,S),
\end{eqnarray*}
which implies
\begin{equation}\label{I1}
 \sum_{v\in V(G)}\mathrm{x}_v\leq \frac{ 2e(L)+2\epsilon e(S)+(1+\epsilon)e(L,S)}{\rho(G)}.
\end{equation}
Let $N_G^c(u):=V(G)\backslash N_G(u)$. By Lemma~\ref{lem-27} and (\ref{I1}),
\begin{equation*}\label{I2}
  \begin{aligned}
|N_G^c(u)| \cdot  \frac{1}{\rho(G)} &\leq \sum_{v\in N_G^c(u)}\mathrm{x}_v
=\sum_{v\in V(G)}\mathrm{x}_v-\sum_{v\in N_G (u)}\mathrm{x}_v=\sum_{v\in V(G)}\mathrm{x}_v-\rho(G) \mathrm{x}_u\\
   &\leq  \frac{ 2e(L)+2\epsilon e(S)+(1+\epsilon)e(L,S)}{\rho(G)}-\rho(G) \mathrm{x}_u.
     \end{aligned}
\end{equation*}
Furthermore, using (\ref{1-1}), (\ref{1}), (\ref{2-1}) and (\ref{2}), we have
\begin{eqnarray*}
  |N_G^c(u)| &\leq&  2e(L)+2\epsilon e(S)+(1+\epsilon)e(L,S)-\rho(G)^2\mathrm{x}_u \\
   &\leq& \epsilon n+\epsilon C_1n+(1+\epsilon)(s+\epsilon)n-s(n-s)(1-\alpha)  \\
   &=& \big[\epsilon(1+C_1)+(1+\epsilon)(s+\epsilon)-s(1-\alpha)\big]n+s^2(1-\alpha)\\
   &\leq& (C_1+s+4)(\alpha+\epsilon)n,
\end{eqnarray*}
where the last inequality holds as long as 
$n\geq {s^2}/ \epsilon$.
Hence, $$d_G(u)=n-|N_G^c(u)|\geq n-(C_1+s+4)(\alpha+\epsilon)n=[1-(C_1+s+4)(\alpha+\epsilon)]n.$$
Denote $C_3:=C_1+s+4>1$, which is independent of $\alpha$ and $\epsilon$. So Claim~2 holds.
\end{proof}

{\bf Claim~3.}
There exist $s$ distinct vertices $v_1, \ldots, v_s\in L$ satisfying
$\mathrm{x}_{v_i}\ge 1-C_4\epsilon$ and $d_G(v_i)\ge (1-C_4\epsilon)n$ for every $i=1,\dots,s$, where $C_4>0$ is a constant independent of $\epsilon$ and $n$.

\begin{proof}
We shall prove this claim  by induction. First  of all, 
setting $v_1=w$, which is a vertex with the largest 
entry of the eigenvector $\bf{x}$, then $\mathrm{x}_{v_1}=1$. Furthermore, by Claim 2, there exists a constant $c_1=C_3>1$ independent of $\epsilon$ and $n$ such that $d_{G}(v_1)\ge (1-c_1\epsilon)n$.

Now assume that we have chosen $v_1, \ldots, v_k\in L$ satisfying
$\mathrm{x}_{v_i}\ge 1-c_k\epsilon$ and $d_G(v_i)\ge (1-c_k\epsilon)n$ for $1\le i\le k$, where $c_k$ is a constant independent of $\epsilon$ and $n$. 
Our goal is to show that there exist an absolute constant $c_{k+1}$ and a vertex $v_{k+1}\notin \{v_1,\ldots ,v_k\}$ 
such that the degree $d(v_{k+1}) \ge 
(1-c_{k+1}\epsilon )n $ and the eigen-entry $\mathrm{x}_{v_{k+1}} \ge 1- c_{k+1}\epsilon$.  

Denote $U=\{v_1, \ldots, v_k\}\subseteq L$. 
By (\ref{1-1}), (\ref{1}) and Claim 1, we have
\begin{eqnarray*}
  s(n-s) &\leq&\rho(G)^2\mathrm{x}_w= 
  \sum\limits_{v\in N(w)}\sum\limits_{z\in N(v)}\mathrm{x}_z\leq\sum\limits_{vz\in E(G)}(\mathrm{x}_v+\mathrm{x}_z)\\
   &=& \sum\limits_{vz\in E(S)}(\mathrm{x}_v+\mathrm{x}_z) +\sum\limits_{vz\in E(L,S)}(\mathrm{x}_v+\mathrm{x}_z)+\sum\limits_{vz\in E(L)}(\mathrm{x}_v+\mathrm{x}_z)\\
    & \leq& 2\epsilon e(S)+2e(L)+\epsilon e(L,S) +\sum_{\substack{ uv\in E(U,S)\\ u\in U }}\mathrm{x}_u 
    + \sum_{\substack{uv\in E(L\backslash U,S)\\ u \in L\backslash U }}\mathrm{x}_u.\\ 
    &\le & \epsilon C_1n+\epsilon n+ \epsilon (s+\epsilon)n+kn+\sum_{\substack{uv\in E(L\backslash U,S) \\ 
    u \in L\backslash U}}\mathrm{x}_u,
\end{eqnarray*}
which implies that
\begin{equation}
\label{claim3-2} 
\sum_{\substack{uv\in E(L\backslash U,S) \\ 
u \in L\backslash U }}\mathrm{x}_u\ge [s-k-\epsilon(C_1+s+2+\epsilon)]n
\end{equation}
as long as $n \geq {s^2}/\epsilon$.
On the other hand, recall that 
$U\subseteq L$ and $V(G)= L\cup S$, then 
$$e(U,S)+e(U,L\backslash U)+2e(U)=\sum_{v\in U}d_G(v)\geq k(1-c_k\epsilon)n,$$
we have
\begin{eqnarray*}
  e(U, S)  &\geq& k(1-c_k\epsilon)n-e(U,L\backslash U)-2e(U) \\
   &\geq& k(1-c_k\epsilon)n-k(|L|-k)-k(k-1) \\
   &\geq&  k(1-c_k\epsilon)n-k(\epsilon n-k)-k(k-1)\\
      &=& k(1-c_k\epsilon-\epsilon)n+k, 
\end{eqnarray*}
where the last inequality holds by  (\ref{upper L}) for sufficiently large $n$. 

By Claim 1, we have
\begin{equation}\label{claim3-3} 
\begin{aligned}
e(L\backslash U,S)=e(L,S)-e(U,S) & \leq(s+\epsilon)n-k(1-c_k\epsilon-\epsilon)n-k \\ 
& <[s+\epsilon-k(1-c_k\epsilon-\epsilon)]n. 
\end{aligned}
\end{equation}

Let $$h(x)=\frac{s-x-\epsilon(C_1+s+2+\epsilon)}{s+\epsilon-x(1-c_k\epsilon-\epsilon)}.$$  It is easy to see that $h(x)$ is decreasing with respect to $1\leq x\leq s-1$.
Then (\ref{claim3-2}) and (\ref{claim3-3}) imply
\begin{eqnarray*}
  \frac{\sum\limits_{\substack{uv\in E(L\backslash U,S) \\ 
  u\in L\backslash U }} \mathrm{x}_u}{e(L\backslash U,S)} &\geq& h(k)\geq h(s-1)
     =\frac{1-\epsilon (C_1+s+2+\epsilon)}{1+\epsilon+(s-1)(c_k\epsilon+\epsilon)}\\
     &\geq&1-(C_1+2s+2)(c_k\epsilon+\epsilon).
\end{eqnarray*}
Hence, by averaging, there exists a vertex $v_{k+1}\in L\backslash U$ such that $$\mathrm{x}_{v_{k+1}}\geq1-(C_1+2s+2)(c_k\epsilon+\epsilon).$$
Therefore, setting $\alpha \le (C_1+2s+2)(c_k\epsilon+\epsilon)$ in Claim~2, we get 
\begin{eqnarray*}
d_G(v_{k+1})&\geq&[1- C_3((C_1+2s+2)(c_k\epsilon+\epsilon)+\epsilon)]n  \\ 
&\geq & [1- C_3(C_1+2s+3)(c_k\epsilon+\epsilon)]n\\
&=&[1-C_3(C_1+2s+3)(c_k+1)\epsilon]n.
\end{eqnarray*}

Let $c_{k+1} :=C_3(C_1+2s+3)(c_k+1)$. 
Then $c_{k+1}$ is independent of $\epsilon$ and $n$. 
Clearly, we have $c_k < c_{k+1}$ since $C_3>1$ obtained from Claim 2. 
Consequently, we get $\mathrm{x}_{v_i}\ge 1-c_{k+1}\epsilon$ and $d_G(v_i)\ge
(1-c_{k+1}\epsilon)n$ for every $i=1,\ldots,k+1$.
Hence  Claim~3 holds.
\end{proof}
Let $v_1,v_2,\ldots,v_s \in L$ be defined in Claim 3. Denote by $A=\{v_1,v_2,\ldots,v_s\}$. 
Let $B:=\cap_{i=1}^s N_G(v_i)$ be the set of common neighbors of vertices of $A$, and $R:=V(G)\backslash (A\cup B)$ be the set of remaining vertices of $G$.
Then by $d_G(v_i)\ge (1-C_4\epsilon)n$ for every $i=1, \ldots, s$, we have
\begin{align*} 
|B|  &\ge  \sum_{i=1}^s |N_G(v_i)| - (s-1) \left|\bigcup_{i=1}^s N_G(v_i) \right| \\ 
&\geq \sum_{i=1}^s(1-C_4\epsilon)n-(s-1)n=(1-C_4s\epsilon)n 
\end{align*}
and
 \begin{equation}  \label{eqeq-R}
 |R|=n-|A|-|B|\leq C_4s\epsilon n.
 \end{equation}

In the sequel, let $G[A,B]$ be a subgraph of $G$ with vertex set $A\cup B$ and edge set $E_G(A,B)$, where $E_G(A,B)$ is the set of edges of $G$ between $A$ and $B$.

\medskip 
{\bf Claim~4.} $A=\{v_1,v_2,\ldots ,v_s\}$ is a clique in $G$.

\begin{proof}
Clearly, $G[A,B]$ is a complete bipartite graph with $|A|=s$ and $|B|=(1-\delta)n$, where $\delta \le C_4s\epsilon$. Moreover,
$(1-3\delta)n\geq2s+1$ for sufficiently large $n$.
Since adding edges to a connected graph strictly increases its spectral radius, by Lemma~\ref{edge-minor-2} and the maximality of $G$, we know that $A$ must induce a clique in $G$. This proves Claim~4.
\end{proof}

{\bf Claim~5.} For every $v\in V(G)\backslash A$, we have $\mathrm{x}_v\leq\frac{1}{C_1+3}$.

\begin{proof}  
On the one hand, 
for any  $u\in R$, 
that is, $u$ is not the common neighbor of vertices of $A$, 
we have $|N_G(u)\cap A| \le s-1$. 
By Lemma \ref{L-minor2} (i), we have $|N_G(u) \cap B| \le 1$. Therefore, 
\begin{equation}\label{e4.9}
|N_{G}(u)\cap(A\cup B)|=|N_{G}(u)\cap A|+|N_{G}(u)\cap B|\leq s-1+1=s.
\end{equation}
Hence, it follows that 
\begin{eqnarray*}
  \rho(G) \sum_{u\in R} \mathrm{x}_u =  \sum_{u\in R} \sum_{w\in N_G(u)} \mathrm{x}_w
   \leq \sum_{u\in R}d_G(u)\leq 2e(R)+e(R,A\cup B)\leq2e(R)+s|R|.
\end{eqnarray*}
Note that $G[R]$ is $F_s$-minor-free. By Lemma \ref{e-minor1}, we have
\begin{equation}\label{e4.10}
 \sum_{u\in R}\mathrm{x}_u \leq\frac{2e(R)+s|R|}{\rho(G)}\leq \frac{C_1|R|+s|R|}{\rho(G)}=\frac{(C_1+s)|R|}{\rho(G)}.
\end{equation}

On the other hand, 
for any vertex $u\in B$, by Lemma \ref{L-minor2} (ii), we have
\begin{equation}\label{e4.11}
|N_{G}(u)\cap(A\cup B)|=|N_{G}(u)\cap A|=s.
\end{equation}
Let $v\in V({G})\backslash A=R\cup B$ 
be a fixed vertex. Next, we will show that 
$\mathrm{x}_v \le \frac{1}{C_1+3}$. 
By (\ref{e4.9}), (\ref{e4.10}) and (\ref{e4.11}), we have $|N_{G}(v)\cap(A\cup B)|\leq s$ and
\begin{eqnarray*}
  \rho(G) \mathrm{x}_v = \sum\limits_{u\in N_G(v)} \mathrm{x}_u\leq\sum_{\substack{u\in N_G(v)\\u\in A\cup B}}\mathrm{x}_u+\sum_{\substack{u\in N_G(v) \\u\in R}}\mathrm{x}_u \leq s+\sum_{u\in R}\mathrm{x}_u
   \leq s+\frac{(C_1+s)|R|}{\rho(G)},
\end{eqnarray*}
which together with (\ref{eqeq-R}) implies 
\begin{eqnarray*}
  \mathrm{x}_v &\leq& \frac{s}{\rho(G)}+\frac{(C_1+s)|R|}{\rho(G)^2}
   \leq\frac{s}{\sqrt{s(n-s)}}+\frac{(C_1+s)C_4\epsilon n}{n-s}\\
   &\leq&\frac{1}{2(C_1+3)}+\frac{1}{2(C_1+3)}
   =\frac{1}{C_1+3},
\end{eqnarray*}
where the last inequality holds as long as 
$\epsilon >0$ is a small constant with 
$(C_1+s)C_4\epsilon (C_1+3)<\frac{1}{4}$, and 
$n$ is sufficiently large satisfying $n\geq {4s(C_1+3)^2+s}$.
So Claim~5 holds.
\end{proof}

{\bf Claim~6.} The induced subgraph $G[B]$ consists of some isolated vertices. 

\begin{proof}
By Lemma \ref{L-minor2} (i), we know that 
  $G[B]$ does not contain 
 a copy of $P_2$, and so  $B$ is an independent set, 
 that is, $G[B]$ consists of some isolated vertices.  
\end{proof}

{\bf Claim~7.} $R$ is empty, and so 
$d_G(v)=n-1$ for any $v\in A$.

\begin{proof}
 Assume that $R$ is not empty. Since $G[R]$ is $F_s$-minor-free, by Lemma \ref{e-minor1}, there is a constant 
 $C_1$ such that 
 $2e(R)\le C_1 |R|$. Then the maximum degree of $G[R]$ 
 is at most $C_1$, and 
 there exists a vertex $v\in R$ such that $
 d_R(v)= |N_G(v)\cap R|\leq C_1$. 
 Now, we can order the vertices of $G[R]$ as follows: $z_1, z_2, \ldots, z_{|R|}$ such that $d_{G[R]}(z_1)\leq C_1$ and 
 for every $i=2,3,\ldots,|R|$,  
 \begin{equation} 
\label{C1}
  |N_{G}(z_i) \cap \{z_{i+1}, \ldots ,z_{|R|}\}| \leq C_1.
  \end{equation}
  In other words, each vertex $z_i\in R$  has at most 
  $C_1$ neighbors in $\{z_{i+1}, \ldots ,z_{|R|}\}$. 
 Recall that $B=\cap_{i=1}^s N_G(v_i)$ 
 and $R=V(G) \setminus (A\cup B)$. 
  Any vertex $z_i \in R$ has at least one non-neighbor in $A$. Moreover, 
  by Lemma \ref{L-minor2} (i), each vertex $z_i\in R$ 
  has at most one neighbor in $B$. 
We define a new graph $G^*$ as below: 
\begin{align*} 
G^* &:=G-\{z_iz_j\in E(G):z_i,z_j\in R\}-\{z_iu\in E(G):z_i\in R,u\in B\} \\ 
&\quad +\{z_iv_j\notin E(G):z_i\in R,v_j\in A\}.
\end{align*}
Clearly, we have $G^*= K_s \vee {I_{n-s}}$. Since $K_s \vee {I_{n-s}}$ is $F_s$-minor-free, 
we know that $G^*$ is also $F_s$-minor-free. 
Using Rayleigh's formula, together with Claims 3 and 5, we obtain 
\begin{eqnarray*}
  \rho(G^*)-\rho(G) &\geq& \frac{{\bf x}^{\mathrm{T}}A(G^*){\bf x}}{{\bf x}^{\mathrm{T}}{\bf x}}-\frac{{\bf x}^{\mathrm{T}} A {\bf x}}{{\bf x}^{\mathrm{T}}{\bf x}} \\
  &\geq&\frac{2}{{\bf x}^{\mathrm{T}}{\bf x}}\left(\sum_{\substack{z_iv_j\notin E(G)\\z_i\in R,v_j\in A}} \mathrm{x}_{v_j}\mathrm{x}_{z_i}-\sum\limits_{\substack{z_iz_j\in E(G)\\z_i,z_j\in R}}\mathrm{x}_{z_i}\mathrm{x}_{z_j}-\sum\limits_{\substack{z_iu\in E(G)\\z_i\in R,u\in B}}\mathrm{x}_{z_i}\mathrm{x}_{u}\right) \\
   &{\geq}&\frac{2}{{\bf x}^{\mathrm{T}}{\bf x}}\left((1-C_4\epsilon)\sum_{i=1}^{|R|}\mathrm{x}_{z_i}-\frac{C_1}{C_1+3}\sum_{i=1}^{|R|}\mathrm{x}_{z_i}-\frac{1}{C_1+3}\sum_{i=1}^{|R|}\mathrm{x}_{z_i}\right)\\
   &=&\frac{2}{{\bf x}^{\mathrm{T}}{\bf x}}\left(1-C_4\epsilon-\frac{C_1+1}{C_1+3}\right)\sum_{i=1}^{|R|}\mathrm{x}_{z_i}>0,
\end{eqnarray*}
where the last inequality holds as long as 
$\epsilon $ is a small positive constant 
so that $\epsilon<\frac{2}{C_4(C_1+3)}$.
Consequently, we get a new graph $G^*$, 
which is an $F_s$-minor-free graph and has larger spectral radius than $G$, a contradiction. Hence, $R$ is empty. This proves Claim~7.
\end{proof}

It follows from Claims~4, 6 and 7 that $G=K_s \vee {I_{n-s}}$, as needed.

\section{Structure of graphs with no $Q_{t}$-minor}

The following result  holds for $Q_t$-minor-free graphs 
by using Lemma \ref{lem-22}.  

\begin{lem}\label{e-minor3-Qt}
Let $G$ be an $n$-vertex $Q_{t}$-minor-free bipartite graph with vertex partition $A$ and $B$. If $|A|=a$ and $|B|=n-a$, then there exists a constant $C>0$ depending only on $t$ such that
$$e(G)\leq Ca+ t n.$$
\end{lem}

\begin{lem}\label{spec1-minor3-ii}
Let $G$ be an $n$-vertex graph with the maximum spectral radius $\rho(G)$ among all $Q_{t}$-minor-free connected graphs. Then $\rho(G)\geq \sqrt{t(n-t)}$. 
\end{lem}

\begin{proof}
Using that $K_{t,n-t}$ is $Q_{t}$-minor-free 
and $\rho (K_{t,n-t})= \sqrt{t(n-t)}$.  
\end{proof}

\begin{lem}\label{L-minor1}
Let $G$ be an $n$-vertex $Q_t$-minor-free graph. If $G$ contains a complete bipartite subgraph $K_{t,(1-\delta)n}=[A,B]$ with $|A|=t$ and $|B|=(1-\delta)n\geq 2t+1$, then \\
(i) $G[B]$ is $P_3$-free, 
and $|N_G(v)\cap B|\leq2$ for any $v\in V(G)\backslash (A\cup B)$; \\
(ii) There are at least $(1-3\delta)n$ vertices in $B$ which have no neighbors in $V(G)\backslash (A\cup B)$.
\end{lem}

\begin{proof}
We first observe that $G[B]$ is $P_3$-free. Indeed, if there exists a path $P_3$ in $G[B]$, then $Q_t$ is a subgraph of $G[A\cup B]$. Hence, $G[B]$ is $P_3$-free. Furthermore, we have the following claim.

\medskip 
{\bf Claim.} If $H$ is a component of $G-(A\cup B)$, then $|N_G(H)\cap B|\leq2$.

\begin{proof}[Proof of Claim]
Suppose that there are three vertices $u,v,w\in N_G(H)\cap B$. Let $G'$ be obtained from $G$ by contracting $G[\{v\}\cup V(H)]$ to a single vertex $v'$. Then $K_{t,(1-\delta)n}= [A,(B\backslash\{v\})\cup\{v'\}]$ and $P_3=  uv'w$ are subgraphs of $G'$. Note that $Q_t$ is a subgraph of the union of $K_{t,(1-\delta)n}= [A,(B\backslash\{v\})\cup\{v'\}]$ and $P_3=  uv'w$. Hence, $Q_t$ is a minor of $G$, which is a contradiction. So the claim holds. 
\end{proof}

By the above Claim, we know that part (i) holds. 
Next, we shall prove (ii).  

Let $R=V(G)\backslash (A\cup B)$ and $D=\{v\in B: N_G(v)\cap R=\emptyset \}$. By the definition of $R$,
$$|R|=n-|A|-|B|=n-t-(1-\delta)n<\delta n,$$
which implies that
$R$ has at most $ \delta n$ components. By the above Claim, $B\backslash D$ has at most $2\delta n$ vertices.
Hence, $$|D|=|B|- |B\backslash D|\geq (1- \delta) n-2\delta n\geq(1- 3\delta) n.$$

This completes the proof.
\end{proof}

\begin{lem}\label{edge-minot-1}
Let $G$ be an $n$-vertex $Q_t$-minor-free graph. Suppose $G$ contains a complete bipartite subgraph $K_{t,(1-\delta)n}=[A,B]$ with $|A|=t$, $|B|=(1-\delta)n$ and $(1-3\delta)n\geq3t+1$. Let $G^*$ be obtained from $G$ by adding edges to $A$ to make it a clique. Then $G^*$ is also $Q_t$-minor-free.
\end{lem}
\begin{proof}
Denote by $R=V(G)\backslash (A\cup B)$ and $D=\{v\in B: N_G(v)\cap R=\emptyset \}$. By Lemma \ref{L-minor1}, $|D|\geq(1- 3\delta) n$. Suppose that $G^*$ contains a $Q_t$-minor. Then there exist $3t+1$ disjoint fragments $S_0,S_1,\ldots,S_{3t}\subseteq V(G^*)=V(G)$ with the following properties:

$(a)$\ $S_0$ and $S_{3j-i}$ are adjacent in $G^*$ for $j=1,\ldots,t$ and $i=0,2$.

$(b)$\ $S_{3j-1}$ and $S_{3j-i}$ are adjacent in $G^*$ for $j=1,\ldots,t$ and $i=0,2$.

$(c)$\ There is an integer $j$ such that $S_j\cap D\neq\emptyset$, $j\in \{0,1,\ldots,3t\}$.

In fact, if $S_0,S_1,\ldots,S_{3t}\subseteq (B\setminus D)\cup R$, then $S_0,S_1,\ldots,S_{3t}$ in $G$ form a $Q_t$-minor, which is a contradiction. Hence, there exists a set $S_j$ such that $S_j\cap (A\cup D)\not=\emptyset$ for $j\in \{0,1,\ldots,3t\}$. If $S_j\cap D\neq\emptyset$, then we are done. Otherwise, we have $S_j\cap A\not=\emptyset$. Furthermore, we can suppose that $S_i\cap D=\emptyset$ for all $0\leq i\leq 3t$. Then choose a vertex $u\in D$ and let $S_j^{\prime}=S_j\cup \{u\}$. Then
$S_0,\ldots,S_j',\dots,S_{3t}$ satisfying $(a), (b)$ and $(c)$.

Let
  $$g(S_0,S_1,\ldots,S_{3t})=|\{S_i:S_i\cap D\neq\emptyset ~\text{for}~ i=0,1,\dots,3t\}|.$$
Now, we choose $3t+1$ disjoint fragments $S_0,S_1,\ldots,S_{3t}$ satisfying $(a), (b)$ and $(c)$ such that $g(S_0,S_1,\ldots,S_{3t})$ is as large as possible.

For $i=0,1,\ldots,3t$, if $|S_i\cap D|\ge 2$, choose a vertex $u_i\in S_i\cap D$ and let $U_i=(S_i\backslash D)\cup\{u_i\}$. If $|S_i\cap D|\leq1$, let $U_i=S_i$.
\begin{claim}\label{claim1}
For $i=0,1,\ldots,3t$, the induced subgraph $G^*[U_i]$ is connected.
\end{claim}
\begin{proof}
For any two vertices $u,v$ in $U_i$, there exists a path $P$ from $u$ to $v$ in $G^*[S_i]$ since $G^*[S_i]$ is connected. If $P$ contains a vertex $w\in(S_i\cap D)\setminus\{u_i\}$, then there exist two vertices $w_1,w_2$ such that they are adjacent to $w$ in $P$. Clearly, $w_1,w_2\notin R$. Since $G[B]$ is $P_3$-free, either $w_1\in A$ or $w_2\in A$. So $w_1$ is adjacent to $w_2$. So there is a path in $G^*[S_i]$ from $u$ to $v$ containing no $w$. By the above analysis, there is a path from $u$ to $v$ in $G^*[U_i]$. Hence, $G^*[U_i]$ is connected, a contradiction.
\end{proof}
\begin{claim}\label{claim2}
If $S_i$ and $S_j$ are adjacent in $G^*$ such that $|U_i\cap D|=1$ and $|U_j\cap D|\leq1$, then $U_i$ and $U_j$ are adjacent in $G^*$.
\end{claim}
\begin{proof}
Suppose on the contrary that there are no edges between $U_i$ and $U_j$ in $G^*$. Then $S_j\cap A=\emptyset$ and all the edges between $S_i$ and $S_j$ in $G^*$ have one endpoint in $(S_i\cap D)\backslash\{u_i\}$ or $(S_j\cap D)\backslash\{u_j\}$. Hence, $S_j\cap B\neq\emptyset$. As $S_j\cap B\neq\emptyset$, we get $S_i\cap A=\emptyset$. Without loss of generality, let $uv$ be one of the above edges with $u\in(S_i\cap D)\backslash\{u_i\}$ and $v\in S_j$. Then $v\in B$. Since $G^*[S_i]$ is connected, $u$ has a neighbor $w$ in $G^*[S_i]$. Then $w\in B$. Hence, $wuv$ is a $P_3$ in $G[B]$, contradicting to Lemma \ref{L-minor1} (ii).
\end{proof}
\begin{claim}\label{claim3}
If $U_i\cap A\neq \emptyset$, then $U_i\cap D\neq \emptyset$ for $0\leq i\leq 3t$.
\end{claim}
\begin{proof}
By Claims \ref{claim1} and \ref{claim2}, disjoint fragments $U_0,U_1, \ldots, U_{3t}$ satisfy $(a), (b)$ and $(c)$. Now suppose that there exists $0\le j\le 3t$ such that $U_j\cap A\neq \emptyset$ and $U_j\cap D= \emptyset$. Then choose a vertex $w\in D\backslash \cup_{i=0}^{3t} U_i$ and let $V_j=U_j\cup\{w\}$ and $V_i=U_i$ for $0\le i\neq j\le 3t$. It is easy to see that the $3t+1$ disjoint fragments $V_0,V_1,\ldots, V_{3t}$ satisfy $(a), (b)$ and $(c)$. Moreover,
$$g(V_0,V_1, \ldots, V_{3t})=g(U_0,U_1, \ldots, U_{3t})+1=g(S_0,S_1, \ldots, S_{3t})+1,$$
which contradicts to the choice of $S_0, \ldots, S_{3t}$.
\end{proof}
\begin{claim}\label{claim4}
$G[U_i]$ is connected for $0\leq i\leq 3t$.
\end{claim}
\begin{proof}
Since $G^*[U_i]$ is connected, there exists a path $P$ from $u$ to $v$ in $G^*[U_i]$ for any two vertices $u,v \in U_i$. If $P$ contains an edge $a_1a_2$ with $a_1, a_2\in A$, then by Claim \ref{claim3}, there exists a vertex $w\in U_i \cap D$. If $w\notin V(P)$, then the edge $a_1a_2$ of $P$ may be replaced by edges $a_1w$ and $a_2w$. If $w\in V(P)$, then the subpath of $P$ containing $a_1a_2$ and $w$ may be replaced by an edge $a_1w$ or $wa_2$. The above transformations yield a path $P'$ from $u$ to $v$ which contains no edge in $G^*[A]$. So there exists a path from $u$ to $v$ in $G[U_i]$ and thus $G[U_i]$ is connected.
\end{proof}
\begin{claim}\label{claim5}
If $S_i$ and $S_j$ are adjacent in $G^*$, then $U_i$ and $U_j$ are adjacent in $G$.
\end{claim}
\begin{proof}
Suppose that there are no edges between $U_i$ and $U_j$ in $G$. By Claim \ref{claim2}, there must exist two vertices $u,v$ such that $u\in U_i\cap A$ and $v\in U_j\cap A$. By Claim \ref{claim3}, there exists a vertex $w\in U_j\cap D$. Hence, there is one edge $uw$ between $U_i$ and $U_j$ in $G$. This is a contradiction.
\end{proof}
By Claims \ref{claim4} and \ref{claim5}, $U_0,U_1, \ldots, U_{3t}$ form a $Q_t$-minor of $G$, which is a contradiction. This completes the proof.
\end{proof}

\section{Proof of Theorem \ref{thm2}} 
\setcounter{equation}{0}

{\bf Proof of Theorem~\ref{thm2}.} Let $G$ be an $n$-vertex $Q_t$-minor-free graph with the maximum spectral radius. 
Thus, the graph $G$ is connected.  
Then let $\mathbf x=(\mathrm{x}_u)_{u\in V(G)}$ be a positive eigenvector of $\rho(G)$ with the maximum entry $\mathrm{x}_w=1$. For $0<\epsilon<1$, denote $$L=\{v\in V(G): \mathrm{x}_v> \epsilon\},$$ 
and 
$$ S=\{v\in V(G): \mathrm{x}_v\leq \epsilon\},$$
where $\epsilon$ will be chosen later. Clearly,
$V(G)=L\cup S$.
By Lemma~\ref{e-minor1}, there is a constant $C'_1$ such that
\begin{equation}\label{5-1}
\begin{aligned}
2e(S)\leq 2e(G)\leq C'_1n.
   \end{aligned}
 \end{equation}
In addition, by Lemma~\ref{spec1-minor3-ii},
\begin{equation}\label{5-2}
\begin{aligned}
\rho(G)\geq \sqrt{t(n-t)}.
   \end{aligned}
 \end{equation}

By the similar methods of Claims~1-3 in the previous proof of Theorem~\ref{thm1}, we can prove the following three claims. For simplicity, we shall omit the detailed proof. 

{\bf Claim~1.} $e(L,S)\leq (t+\epsilon)n$ and $2e(L)\leq \epsilon n$.

{\bf Claim~2.}
If $u\in L$ is a vertex with $\mathrm{x}_u=1-\alpha$ 
for some $\alpha\in (0,1)$, then there exists a constant $C'_3 >0$ independent of $\alpha$ and $\epsilon$ such that
$ d_G(u)\geq [1- C'_3(\alpha+\epsilon)]n.$

{\bf Claim~3.}
There exist $t$ distinct vertices $v_1, \ldots, v_t\in L$ satisfying
$\mathrm{x}_{v_i}\ge 1-C'_4\epsilon$ and $d_G(v_i)\ge (1-C'_4\epsilon)n$ for $i=1,\dots,s$, where $C'_4 >0$ is a constant independent of $\epsilon$ and $n$.

Let $v_1,v_2,\ldots,v_t \in L$ be defined in Claim 3. Denote by $A=\{v_1,v_2,\ldots,v_t\}$,  $B=\cap_{i=1}^t N_G(v_i)$ and $R=V(G)\backslash (A\cup B)$.
Then by $d_G(v_i)\ge (1-C'_4\epsilon)n$ for $i=1, \ldots, t$, we have
$$|B|\geq \sum_{i=1}^t(1-C'_4\epsilon)n-(t-1)n=(1-C'_4t\epsilon)n$$
and
$$|R|=n-|A|-|B|\leq C'_4t\epsilon n.$$

{\bf Claim~4.} $A$ is a clique.

\begin{proof}
Clearly, $G[A,B]$ is a complete bipartite graph with $|A|=t$ and $|B|=(1-\delta)n$, where $\delta \le C'_4t\epsilon$. Moreover,
$(1-3\delta)n\geq3t+1$ for sufficiently large $n$.
Since adding edges to a connected graph strictly increases its spectral radius, by Lemma~\ref{edge-minot-1} and the maximality of $G$, $A$ must induce a clique in $G$. This proves Claim~4.
\end{proof}

{\bf Claim~5.} For $v\in V(G)\backslash A$, 
we have $\mathrm{x}_v\leq\frac{1}{C'_1+3}$.

\begin{proof}
For any vertex $v\in R$, by Lemma \ref{L-minor1} (i), we have
\begin{equation}\label{5.14}
|N_{G}(v)\cap(A\cup B)|=|N_{G}(v)\cap A|+|N_{G}(v)\cap B|\leq t-1+2=t+1.
\end{equation}
Hence,
\begin{eqnarray*}
  \rho(G) \sum_{v\in R} \mathrm{x}_v &=&\sum_{v\in R} \sum_{u\in N_G(v)} \mathrm{x}_u
   \leq \sum_{v\in R}d_G(v)\leq 2e(R)+e(R,A\cup B)\leq2e(R)+(t+1)|R|.
\end{eqnarray*}
Note that $G[R]$ is $Q_t$-minor-free. By Lemma \ref{e-minor1}, we have
\begin{equation}\label{5.15}
 \sum_{v\in R}\mathrm{x}_v\leq\frac{2e(R)+(t+1)|R|}{\rho(G)}\leq \frac{C'_1|R|+(t+1)|R|}{\rho(G)}=\frac{(C'_1+t+1)|R|}{\rho(G)}.
\end{equation}

For any vertex $v\in B$, by Lemma \ref{L-minor1} (ii), we have
\begin{equation}\label{5.16}
|N_{G}(v)\cap(A\cup B)|=|N_{G}(v)\cap A|+|N_{G}(v)\cap B|\leq t+1.
\end{equation}
Let $v\in V({G})\backslash A=R\cup B$. By (\ref{5.14}), (\ref{5.15}) and (\ref{5.16}), we have $|N_{G}(v)\cap(A\cup B)|\leq t+1$ and
\begin{align*}
  \rho(G) \mathrm{x}_v = \sum\limits_{uv\in E(G)} \mathrm{x}_u &=\sum_{\substack{u\in A\cup B\\ 
  uv\in E(G)}}\mathrm{x}_u+\sum_{\substack{u\in R \\ 
  uv\in E(G)}}\mathrm{x}_u \\
  & \leq t+1+\sum_{u\in R}\mathrm{x}_u
     \leq t+1+\frac{(C'_1+t+1)|R|}{\rho(G)},
\end{align*}
which implies that
\begin{eqnarray*}
  \mathrm{x}_v \leq \frac{t+1}{\rho(G)}+\frac{(C'_1+t+1)|R|}{\rho(G)^2}
   &\leq & \frac{t+1}{\sqrt{t(n-t)}}+\frac{(C'_1+t+1)C'_4\epsilon n}{n-t}\\
   &\leq&\frac{1}{2(C'_1+3)}+\frac{1}{2(C'_1+3)}
   =\frac{1}{C'_1+3},
\end{eqnarray*}
where the last inequality holds as long as 
$\epsilon$ is a small constant such that 
$(C_1' +t+1)C_4'\epsilon (C_1' +3) < \frac{1}{4}$, and 
$n$ is sufficiently large with 
$n\geq {4(1+1/t)^2(C'_1+3)^2} +t$.
So Claim~5 holds.
\end{proof}

{\bf Claim~6.} $G[B]$ consists of independent edges and isolated vertices. 

\begin{proof}
It is straightforward by using Lemma \ref{L-minor1} (i). 
\end{proof}

{\bf Claim~7.} $R$ is empty, and  $d_G(v)=n-1$ for any $v\in A$.

\begin{proof}
 Assume that $R$ is not empty. Since $G[R]$ is $Q_t$-minor-free, by Lemma \ref{e-minor1}, there exists a vertex $v\in R$ such that $|N_G(v)\cap R|\leq C'_1$. Order the vertices of $G[R]$ as follows: $z_1, z_2, \ldots, z_{|R|}$ such that $d_{G[R]}(z_1)\leq C'_1$ and $d_{G[R\backslash\{z_1, z_2, \ldots, z_{i-1}\}]}(z_i)\leq C'_1$ for $i=2,\ldots,|R|$. By the definition of $B$, any vertex in $R$ can not be adjacent to all vertices in $A$. Let 
\begin{align*}
G^* &=G-\{z_iz_j\in E(G):z_i,z_j\in R\}-\{z_iu\in E(G):z_i\in R,u\in B\} \\ 
& \quad +\{z_iv_j\notin E(G):z_i\in R,v_j\in A\}.
\end{align*}
Clearly, $G^*$ is a subgraph of $K_t\vee M_{n-t}$. Since $K_t\vee M_{n-t}$ is $Q_t$-minor-free, $G^*$ is also $Q_t$-minor-free. By Lemma \ref{L-minor1} (i), Claims 3 and 5,
\begin{eqnarray*}
  \rho(G^*)-\rho(G) &\geq& \frac{{\bf x}^{\mathrm{T}}A(G^*){\bf x}}{{\bf x}^{\mathrm{T}}{\bf x}}-\frac{{\bf x}^{\mathrm{T}} A {\bf x}}{{\bf x}^{\mathrm{T}}{\bf x}} \\
  &\geq&\frac{2}{{\bf x}^{\mathrm{T}}{\bf x}}\left(\sum_{\substack{z_iv_j\notin E(G)\\z_i\in R,v_j\in A}} \mathrm{x}_{v_j}\mathrm{x}_{z_i}-\sum\limits_{\substack{z_iz_j\in E(G)\\z_i,z_j\in R}}\mathrm{x}_{z_i}\mathrm{x}_{z_j}-\sum\limits_{\substack{z_iu\in E(G)\\z_i\in R,u\in B}}\mathrm{x}_{z_i}\mathrm{x}_{u}\right) \\
   &\geq&\frac{2}{{\bf x}^{\mathrm{T}}{\bf x}}\left((1-C'_4\varepsilon)\sum_{i=1}^{|R|}\mathrm{x}_{z_i}-\frac{C'_1}{C'_1+3}\sum_{i=1}^{|R|}\mathrm{x}_{z_i}-\frac{2}{C'_1+3}\sum_{i=1}^{|R|}\mathrm{x}_{z_i}\right)\\
   &=&\frac{2}{{\bf x}^{\mathrm{T}}{\bf x}}\left(1-C'_4\epsilon-\frac{C'_1+2}{C'_1+3}\right)\sum_{i=1}^{|R|}\mathrm{x}_{z_i}>0,
\end{eqnarray*}
where the last inequality holds as long as $\epsilon<\frac{1}{C'_4(C'_1+3)}$.
Then $G^*$ is a $Q_t$-minor-free graph with larger spectral radius, a contradiction. Hence, $R$ is empty. This proves Claim~7.
\end{proof}

It follows from Claims~4, 6 and 7 that $G$ is a subgraph of $K_t\vee M_{n-t}$. Note that adding edges to a connected graph will  increase spectral radius strictly. By the maximality of $G$, we know that $G$ must be $K_t\vee M_{n-t}$.

\section{Concluding remarks}

As we stated in the introduction, 
the spectral extremal problem for 
$F_s$-free graphs  was completely studied in \cite{cioaba2,ZLX2022}. 
In this paper, we have investigated the $F_s$-minor-free graphs.   
As we all know, it is challenging and difficult to treat the  extremal problem 
when we forbid bipartite graphs as substructures.  
Recall that $Q_t$ is the graph obtained from $t$ copies of $C_4$ by intersecting one vertex. Clearly, $Q_t$ is a bipartite graph. 
In Theorem \ref{thm2}, we have proved that 
$K_t\vee M_{n-t}$ attains the maximal spectral radius 
over all $Q_t$-minor-free graphs of order $n$.  
However, there is no result on the spectral extremal problem 
for $Q_t$-free graphs, although there are several papers involving the intersecting odd cycles \cite{LP2021,CLZ2021,WZ2022dm}. 
Inspired by Theorem \ref{thm2}, we propose the following 
problem for interested readers.

\begin{pb} \label{pb-71}
Let $t\ge 1$ and $n$ be sufficiently large. 
If $G$ is a $Q_t$-free graph on $n$ vertices, then 
$$ \rho (G) \le \rho (K_t \vee M_{n-t}), $$
equality holds if and only if  $G=K_t \vee M_{n-t}$. 
\end{pb}

We remark here that the case $t=1$ in Problem \ref{pb-71} 
reduces to the problem for $C_4$-free graphs, it 
was early proved by Nikiforov \cite{NikiforovKr} 
for odd $n$, and by Zhai and Wang \cite{zhaic4} for even $n$.  
Unlike the spectral extremal graphs among $F_s$-free and/or $F_s$-minor-free graphs, 
it seems possible from Problem \ref{pb-71}  that the spectral extremal graphs among $Q_t$-free and/or  $Q_t$-minor-free graphs are the same.  Hence, a natural question one may ask is that for which type of graphs $H$, 
the  spectral extremal graph over all $H$-free graphs is the same as that over all  $H$-minor-free graphs.

To our knowledge, the spectral extremal problems 
are investigated until now for $K_r$-minor-free, $K_{s,t}$-minor-free, $F_s$-minor-free and 
$Q_t$-minor-free graphs. 
It is also important for us to consider 
the spectral problem for $H$-minor-free graphs 
when $H$ is some particular graph, 
such as,  books, wheels, fans, cycles, intersecting cycles, intersecting cliques or disjoint cliques, etc. 

\section*{Declaration of competing interest}

The authors declare that they have no conflicts of interest to this work.

\section*{Acknowledgments}
This work was supported by NSFC (Nos. 12271527, 12071484, 11931002), Natural Science Foundation
of Hunan Province (Nos. 2020JJ4675, 2021JJ40707) and the Fundamental Research Funds for the Central Universities of Central South University (Grant No. 2021zzts0034).

{
}
\end{document}